\documentclass{amsart}
\usepackage{latexsym}
\usepackage{amstext}
\usepackage{amsmath}
\usepackage{amssymb}
\usepackage{amsthm}
\usepackage{url}

\newtheorem{theorem}{Theorem}

\newtheorem{corollary}[theorem]{Corollary}
\newtheorem{conjecture}[theorem]{Conjecture}

\theoremstyle{definition}

\def \co {\mathcal{O}}
\def \pt {\rm{pt}}
\newcommand{\Ok}{{\mathcal O}_k}
\def \kbar {\overline{k}}

\begin{document}
\bibliographystyle{amsplain}
\title[The exceptional set in Vojta's conjecture for algebraic points]{The exceptional set in Vojta's conjecture for algebraic points of bounded degree}
\author{Aaron Levin}
\address{Department of Mathematics\\Michigan State University\\East Lansing, MI 48824}
\email{adlevin@math.msu.edu}
\date{}

\begin{abstract}
We study the dependence on various parameters of the exceptional set in Vojta's conjecture.  In particular, by making use of certain elliptic surfaces, we answer in the negative the often-raised question of whether Vojta's conjecture holds when extended to all algebraic points (that is, if the conjecture holds without fixing a bound on the degree of the algebraic points).
\end{abstract}
\maketitle
\
\section{Introduction}

Inspired by results in Nevanlinna theory, Vojta \cite{Vo, Vo2} made the following deep conjecture in Diophantine approximation for algebraic points of bounded degree in a nonsingular complete variety.
\begin{conjecture}[Vojta's Conjecture]
\label{Vgeneral}
Let $X$ be a nonsingular complete variety with canonical divisor $K$.  Let $D$ be a normal crossings divisor on $X$, and let $k$ be a number field over which $X$ and $D$ are defined.  Let $S$ be a finite set of places of $k$, $A$ a big divisor on $X$, $\epsilon>0$, and $r$ a positive integer.  Then there exists a proper Zariski-closed subset $Z=Z(r,\epsilon,X,D,A,k,S)$ of $X$ such that
\begin{equation}
\label{Vineq}
m_{D,S}(P)+h_K(P)\leq d(P)+\epsilon h_A(P)+O(1)
\end{equation}
for all points $P\in X(\kbar)\setminus Z$ with $[k(P):k]\leq r$.
\end{conjecture}
Here $m_{D,S}$ is a proximity function, measuring the part of the height relative to $D$ coming from places lying above places in $S$, $h_K$ and $h_A$ are absolute logarithmic height functions associated to the divisors $K$ and $A$, respectively, and $d(P)=\frac{1}{[k(P):\mathbb{Q}]}\log |D_{k(P)/\mathbb{Q}}|$ is the absolute logarithmic discriminant.  The set $Z$ is called the {\it exceptional set}.  We refer the reader to \cite{Vo} for a more thorough discussion of the definitions and properties of the objects appearing in the conjecture.

Vojta's Conjecture has a wide range of important consequences (see \cite{Vo}), of which we just mention the $abc$-conjecture and the Bombieri-Lang Conjecture on rational points on varieties of general type.  Various strengthened versions of Vojta's Conjecture have frequently appeared in the literature, often asserting that the exceptional set $Z$ in the conjecture can be chosen independently of some of the parameters appearing in the conjecture.  It is this aspect of Vojta's Conjecture that we wish to study here, that is, the dependence of $Z$ on the parameters $r,\epsilon,X,D,A,k,$ and $S$.

That the exceptional set $Z$ must depend on $X$ and $D$ is trivial (indeed, $m_{D,S}(P)$ is not even well-defined for points in the support of $D$).  The first nontrivial general observation regarding the dependence of $Z$ on the parameters is due to Vojta himself, who showed in \cite{Vo3} that, in general, the set $Z$ must be allowed to depend on the parameter $\epsilon$.  Vojta's example in \cite{Vo3} used quadratic points (i.e., $r=2$).  In \cite{LMW}, it was shown that $Z$ must be allowed to depend on $\epsilon$ even when the conjecture is restricted to rational points ($r=1$).  Of course, in specific cases, it may happen that $Z$ can be chosen independently of $\epsilon$.  For instance, in Schmidt's Subspace Theorem ($r=1$, $X=\mathbb{P}^n$, $D$ a sum of hyperplanes in general position), one of the few cases where Vojta's Conjecture is known,  Vojta showed \cite{Vo3} that the exceptional set $Z$ (which can be taken to consist only of hyperplanes) can in fact be taken to be independent of the parameter $\epsilon$.  We note also that the dependence of $Z$ on $\epsilon$ has ramifications for the form of the error term in \eqref{Vineq}.  In particular, contrary to a conjecture of Lang \cite{Lan} for the error term in Roth's theorem ($X=\mathbb{P}^1$, $r=1$), in higher dimensions one cannot replace the term $\epsilon h_A(P)$ in \eqref{Vineq} by a term like $\log h_A(P)$.  In fact, if one considers Vojta's conjecture with ``truncated counting functions", by a result of Stewart and Tijdeman \cite{Tij} (see also \cite{vF}), this same phenomenon occurs even in the one-dimensional case, where the error term is essentially the same as the error term in the $abc$ conjecture (see \cite{vF2}).

In Section \ref{depA}, we show that $Z$ must be allowed to depend on the big divisor $A$.  In allowing arbitrary big divisors $A$, this is hardly surprising since the associated height $h_A$ may not even be positive for points in the base locus of the linear system $|A|$.  However, we will show that even if one restricts to ample divisors $A$ and rational points ($r=1$), the set $Z$ must be allowed to depend on $A$.  Since for any two ample divisors $A$ and $A'$ on a variety $X$ we have $h_A\gg\ll h_{A'}$, this phenomenon is necessarily closely related to the dependence of $Z$ on $\epsilon$.

We have nothing new to offer regarding the dependence of $Z$ on the number field $k$ or the set of places $S$.  If $Z$ could not be taken independent of $k$ and $S$, this would imply that the exceptional set $Z$ is of a highly nongeometric nature, violating the spirit of Vojta's Conjecture.  Consequently, Conjecture \ref{Vgeneral} is frequently stated with the additional hypothesis that $Z$ can be chosen independently of $k$ and $S$.  Adopting this point of view, from now on we will notationally omit the (possible) dependence of $Z$ on $k$ and $S$.

Finally, we will give examples showing that the exceptional set $Z$ in Vojta's Conjecture must be allowed to depend on the parameter $r$.  Explicitly, we will show this for a large class of elliptic surfaces.
\begin{theorem}
\label{tsem}
Let $X$ be a semistable elliptic surface over $\mathbb{P}^1$ with a section over a number field $k$.  Let $D$ be a sum of fibers of $X$ containing the singular fibers.  If Vojta's conjecture holds for $X$ and $D$, then the exceptional set $Z(r,\epsilon,X,D,A)$ must depend on the parameter $r$.
\end{theorem}
We can also produce similar examples (with a different, but related, class of surfaces $X$) where $D=0$ (Theorem \ref{D0}).

As a consequence of Theorem \ref{tsem}, we obtain that Vojta's Conjecture does not hold uniformly for all algebraic points, answering a question that has frequently been raised (e.g., \cite[p.\ 64]{Vo} or \cite[p.\ 223]{Lan2}).
\begin{corollary}
\label{calg}
Vojta's Conjecture does not hold when extended to $r=\infty$ (i.e., to all algebraic points), even when allowing the $O(1)$ term to depend on $[k(P):k]$.
\end{corollary}
This is somewhat surprising since the analogous function field version of Vojta's Conjecture has been proved for curves (independently by McQuillan \cite{McQ} and Yamanoi \cite{Yam}) with $r=\infty$.  Moreover, previous results in the equidimensional Nevanlinna theory of coverings of $\mathbb{C}^n$ \cite{Ch},\cite[Ch. IV]{LC}  also appeared to support the possibility that Vojta's Conjecture would hold uniformly for all algebraic points.  

Despite Corollary \ref{calg}, it is still quite plausible that Vojta's Conjecture may hold uniformly for all algebraic points after making some simple modifications to the terms in the inequality \eqref{Vineq}.  One possibility is given by the original form of the conjecture proposed in \cite{Vo}.  In \cite{Vo}, Vojta included a $\dim X$ factor in the conjecture in front of the discriminant term $d(P)$, but this factor was later thought to be unnecessary (see, e.g., the comment after Conjecture 2.1 in \cite{Vo2}).  None of the examples presented here prevent the possibility that Conjecture \ref{Vgeneral} could hold for all algebraic points with a $(\dim X) d(P)$ term in the inequality \eqref{Vineq} (instead of $d(P)$).  More precisely, our examples allow the possibility that Conjecture \ref{Vgeneral} could hold for all algebraic points when the $d(P)$ term is replaced by $\frac{3}{2}d(P)$ in the inequality \eqref{Vineq}.  At the end of Section \ref{sell} we will also discuss some recent examples \cite{ACG} of Autissier, Chambert-Loir, and Gasbarri concerning a geometric analogue of \eqref{Vineq}, which may provide further insight into the correct form of a uniform version of Vojta's Conjecture.

\subsection*{Acknowledgments}
I thank Antoine Chambert-Loir for many fruitful discussions and for introducing me to and explaining to me his joint paper \cite{ACG} with Autissier and Gasbarri (which was an inspiration for the present paper).  I thank William Cherry for helpful comments on an earlier draft, and in particular, for bringing to my attention the paper \cite{Tij}.  I thank Tom Tucker and Paul Vojta for further useful discussions.  I thank the Institute for Advanced Study for providing a nice research environment in which some of this research was done, supported by NSF grant DMS-0635607.

\section{The Dependence of the Exceptional Set on $A$}
\label{depA}
We now show that the exceptional set must be allowed to depend on the divisor $A$.
\begin{theorem}
The exceptional set $Z=Z(r,\epsilon,X,D,A)$ cannot be chosen independently of $A$, even for $r=1$, $D=0$, and varying only over ample divisors $A$.
\end{theorem}
As mentioned earlier, this theorem is very closely related to the fact that there are examples where the exceptional set $Z(r,\epsilon, X, D, A)$ must depend on $\epsilon$.
\begin{proof}
Let $X$ be a nonsingular projective surface over a number field $k$ with infinitely many distinct curves $C\cong \mathbb{P}^1_k$ having intersection number $K_X.C>0$.  An example of such a surface was explicitly given in \cite[Ex. 2.4]{LMW}.  Let $0<\epsilon<\frac{1}{2}$.  It suffices to show that for any such curve $C$, there exists an ample divisor $A$ with $C\subset Z(r=1,\epsilon, X,D=0,A)$.  Let $C$ be such a curve.  We first claim that there exists an ample divisor $A'$ on $X$ satisfying $\gcd(A'.C,C.C)\leq 2$.  Note that since $C\cong \mathbb{P}^1$, by the adjunction formula, $K_X.C=-2-C.C$ and $C.C<0$.  It follows that for any ample divisor $A''$ on $X$, either $A'=A''$ or $A'=nA''+K_X$, for some sufficiently large integer $n$, will satisfy $A'$ is ample and $\gcd(A'.C,C.C)\leq 2$.  Since $A'.C>0$ and $C.C<0$, we may choose positive integers $i$ and $j$ such that $i(A'.C)+j(C.C)=\gcd(A'.C,C.C)\leq 2$.  Let $A=iA'+jC$.  We claim that $A$ is ample.  By construction, we have $A.C>0$.  For any curve $C'\neq C$, $C.C'\geq 0$ and $A'.C'>0$, so $A.C'>0$.  Moreover, since $A'.A>0$ and $C.A>0$, we have $A^2>0$.  It follows from the Nakai-Moishezon criterion that $A$ is ample.  Since $K_X.C-\epsilon A.C\geq 1-2\epsilon>0$, infinitely many points $P$ in $C(k)$ will satisfy $h_K(P)>\epsilon h_A(P)$.  So we must have $C\subset Z(r=1,\epsilon, X,D=0,A)$, as was to be shown.
\end{proof}

\section{Elliptic Surfaces and the Dependence of the Exceptional Set on $r$}
\label{sell}
In this section we will study certain instances of Vojta's Conjecture on elliptic surfaces.  We assume throughout that our (nonsingular, projective) elliptic surfaces are nonconstant and relatively minimal (i.e., no fiber contains an exceptional curve of the first kind).  Recall that such an elliptic surface is called {\it semistable} if every singular fiber has Kodaira type $I_n$ for some $n\geq 1$.  In particular, on a semistable elliptic surface $X$, every fiber is a normal crossings divisor.

We will find it convenient to consider a modified version of Vojta's Conjecture where the coefficient of $d(P)$ is a function of the dimension of $X$.  Let $f:\mathbb{N}\to \mathbb{R}$ be a fixed positive real-valued function.
\begin{conjecture}[Modified Vojta Conjecture with function $f(n)$]
\label{Vgeneral2}
Let $X$ be a nonsingular complete variety of dimension $n$ with canonical divisor $K$.  Let $D$ be a normal crossings divisor on $X$, and let $k$ be a number field over which $X$ and $D$ are defined.  Let $S$ be a finite set of places of $k$, $A$ a big divisor on $X$, $\epsilon>0$, and $r$ a positive integer.  Then there exists a proper Zariski-closed subset $Z=Z(f,r,\epsilon,X,D,A)$ of $X$ such that
\begin{equation}
\label{Vineq2}
m_{D,S}(P)+h_K(P)\leq f(n)d(P)+\epsilon h_A(P)+O(1)
\end{equation}
for all points $P\in X(\kbar)\setminus Z$ with $[k(P):k]\leq r$.
\end{conjecture}
We now prove the main result of this section.
\begin{theorem}
\label{main}
Let $X$ be a semistable elliptic surface over $\mathbb{P}^1$ with a section over a number field $k$.  Let $D$ be a sum of an even number $N$ of fibers, containing the singular fibers.  If Conjecture \ref{Vgeneral2} holds for $X$, $D$, and a function $f$ satisfying $f(2)<1+\frac{\chi(\co_X)}{N-2}$, then the exceptional set $Z(f,r,\epsilon,X,D,A)$ must depend on the parameter $r$.
\end{theorem}

\begin{proof}
By enlarging $k$, we may assume that every fiber in $D$ is defined over $k$ and that $\Ok^*$ is infinite.  Let $\pi:X\to \mathbb{P}^1$ be the fibration and let $\sigma:\mathbb{P}^1 \to X$ be a section over $k$ (which we also identify with its image).  We first claim that $\sigma\setminus D\cong \mathbb{P}^1\setminus \{N \text{ points}\}$ contains an infinite set of integral points $T$ with $[k(P):k]\leq \frac{N}{2}$ for all $P\in T$.  Let $R=\sigma \cap D=R_1\cup R_2$, where $R_1$ and $R_2$ are disjoint sets each with $\frac{N}{2}$ elements of $R$.  Let $\psi$ be a rational function over $k$ on $\mathbb{P}^1$ having zeros exactly at points in $R_1$ and poles exactly at points in $R_2$.  Then $\psi$ defines a morphism $\psi:\sigma\setminus D\to \mathbb{G}_m$ of degree $\frac{N}{2}$.  Since $\mathbb{G}_m$ contains an infinite set $T'$ of integral points over $k$ ($\Ok^*$ was assumed infinite), the set $T=\psi^{-1}(T')$ gives an infinite set of integral points on $\sigma\setminus D$ with $[k(P):k]\leq \frac{N}{2}$ for all $P\in T$.

Taking $\sigma$ to be the zero section on $X$, we have a rational map $[m]:X\to X$ induced by the fiberwise multiplication-by-$m$ map on the nonsingular fibers of $X$.  Let $C\neq \sigma$ be an irreducible component of the Zariski-closure of $[m]^{-1}(\sigma\setminus D)$.  Let $L\supset k$ be a number field that is a field of definition for $C$.  Then $[m]$ induces a morphism (over $L$) $\phi:C\to \sigma$.  Let $\tilde{T}=\phi^{-1}(T)$.  Let $A=\sigma+F$ where $F$ is some fiber.  It is easy to see that $A$ is a big divisor on $X$.  Let $S$ be the set of archimedean places of $k$.  We claim that for any $\epsilon>0$ and any constant $c$, all but finitely many points $P$ in $\tilde{T}$ satisfy
\begin{equation}
\label{fundineq}
m_{D,S}(P)+h_K(P)> \left(1+\frac{\chi(\co_X)-2\epsilon}{N-2}\right)d(P)+\epsilon h_A(P)+c.
\end{equation}
Let us first show that this claim implies the theorem.  Every point $P$ in $\tilde{T}$ satisfies $[k(P):k]\leq \frac{N}{2}[L:k]\deg \phi$.  Since $f(2)<1+\frac{\chi(\co_X)}{N-2}$, it follows from \eqref{fundineq} that for sufficiently small $\epsilon$ (depending only on $f(2)$), $C\subset Z(f,r,\epsilon,X,D,A)$, where $r=r(C)=\frac{N}{2}[L:k]\deg \phi$.  As we vary the integer $m$, we obtain infinitely many distinct such curves $C$ (with different values $r=r(C)$), obtaining the desired result.

We now show \eqref{fundineq}.  First, since $\tilde{T}=\phi^{-1}(T)$ and $T$ is a set of integral points on $\sigma\setminus D$, it follows that $\tilde{T}$ is a set of integral points on $C\setminus D$.  Thus,
\begin{equation*}
m_{D,S}(P)=h_D(P)+O(1), \quad \forall P\in \tilde{T}.
\end{equation*}
The canonical divisor is given by $K_X=(\chi(\co_X)-2)F$ \cite[Th. 12]{Kod}, where $F$ is a fiber.  We have $h_A=h_\sigma+h_F+O(1)$.  It is known that $\sigma\cap C=\emptyset$ (see, e.g., \cite[p. 72, Prop.\ (VII.3.2)]{Mir}).  It follows that $h_\sigma$ is bounded on $C$.  So
\begin{equation*}
h_A=h_F+O(1), \quad \forall P\in C.
\end{equation*}
Since $D\sim NF$, it follows that for all $P$ in $\tilde{T}$ we have
\begin{equation*}
m_{D,S}(P)+h_K(P)-\epsilon h_A(P)=(N+\chi(\co_X)-2-\epsilon)h_F(P)+O(1).
\end{equation*}
Note that $F|_{C}$ is ample, and so for any infinite set of points of bounded degree on $C$ we have $h_F(P)\to \infty$.  Thus, for any constant $c$, all but finitely many points $P$ in $\tilde{T}$ satisfy
\begin{equation*}
m_{D,S}(P)+h_K(P)-\epsilon h_A(P)>(N+\chi(\co_X)-2-2\epsilon)h_F(P)+c.
\end{equation*}
Comparing with \eqref{fundineq}, to finish the proof it suffices to show that
\begin{equation}
\label{ineq1}
d(P)\leq (N-2)h_F(P)+O(1), \quad \forall P\in \tilde{T}.
\end{equation}
This will follow from the following two results.
\begin{theorem}[Mahler \cite{Mah}]
\label{dineq1}
The inequality
\begin{equation*}
d(P)\leq (2r-2)h(P)+O(1)
\end{equation*}
holds for all $P\in \mathbb{P}^1(\kbar)$ with $[k(P):k]\leq r$.
\end{theorem}
\begin{theorem}[Chevalley-Weil]
\label{dineq2}
Let $\phi:X\to Y$ be a finite morphism of nonsingular projective varieties, with ramification divisor $R$ on $X$.  Let $T$ be a set of integral points of bounded degree on $X\setminus R$.  Then
\begin{equation*}
d(P)\leq d(\phi(P))+O(1)
\end{equation*}
for all $P$ in $T$.
\end{theorem}
For a more general quantitative version of the Chevalley-Weil theorem, see \cite[Th. 5.1.6]{Vo}.

Since $[k(P):k]\leq \frac{N}{2}$ for points $P$ in $T$, by Theorem \ref{dineq1} we have
\begin{equation}
\label{ineq2}
d(P)\leq (N-2)h(P)+O(1), \quad \forall P\in T.
\end{equation}
The map $\phi:C\to \sigma$ is unramified outside of $D|_C$.  So by Theorem \ref{dineq2}
\begin{equation}
\label{ineq3}
d(Q)\leq d(\phi(Q))+O(1), \quad \forall Q\in \tilde{T}.
\end{equation}
By functoriality, $h_F(P)=h(\pi(P))+O(1)$ for all $P\in X(\kbar)$.  So
\begin{equation}
\label{ineq4}
h_F(Q)=h_F(\phi(Q))+O(1)=h(\phi(Q))+O(1), \quad \forall Q\in C(\kbar).
\end{equation}
Then combining \eqref{ineq2}, \eqref{ineq3}, and \eqref{ineq4}, we obtain \eqref{ineq1} as desired.
\end{proof}

To optimally apply Theorem \ref{main}, we want to minimize the number of singular fibers on a semistable elliptic surface relative to $\chi(\co_X)$.  This problem was studied by Shioda in \cite{Shi}.

\begin{theorem}[Shioda]
\label{Sh}
Let $X$ be a semistable elliptic surface over $\mathbb{P}^1$ with a section.  Let $N$ be the number of singular fibers.  Then $N\geq 2\chi(\co_X)+2$.  Furthermore, for any positive integer $\chi$, there exists such an $X$ with a section over a number field $k$, $\chi(\co_X)=\chi$, and exactly $N=2\chi+2$ singular fibers.
\end{theorem}
The construction of semistable elliptic surfaces with a minimal number of singular fibers is related to so-called Davenport-Stothers triples and to polynomials which achieve equality in the polynomial $abc$-theorem.  As a consequence of Shioda's result, we obtain:

\begin{corollary}
\label{mainc}
If Conjecture \ref{Vgeneral2} holds for $r=\infty$ then we must have $f(2)\geq \frac{3}{2}$.
\end{corollary}

A natural question to ask is whether Theorem \ref{main} or Corollary \ref{mainc} can be improved by looking at elliptic surfaces $X$ (with a section) over a higher genus base curve $B$.  In general, for such an $X$ and a base curve $B$, we have Kodaira's formula \cite[Th. 12]{Kod}
\begin{equation*}
K_X=\pi^*(K_B+E)
\end{equation*}
where $E$ is some divisor of degree $\chi(\co_X)$ on $B$, and an inequality of Song and Tucker \cite[Eq. (2.0.3)]{ST}
\begin{equation*}
d(P)\leq (2r+\epsilon)h_{\pt}(P)+h_{K_B}(P)+O(1)
\end{equation*}
for all  $P$ in $B(\kbar)$ with $[k(P):k]\leq r$, where $\pt$ is any point of $B(k)$.  Mimicking the proof of Theorem \ref{main}, we find that the key quantity is the ratio 
\begin{equation*}
\frac{N+2g-2+\chi(\co_X)}{N+2g-2}=1+\frac{\chi(\co_X)}{N+2g-2},
\end{equation*}
where $N$ is the smallest even integer greater than or equal to the number of singular fibers of $X$.  Unfortunately, we have the general inequality
\begin{equation*}
N\geq 2\chi(\co_X)+2-2g.
\end{equation*}
Thus, it does not appear that one can improve on the constant $\frac{3}{2}$ in Corollary \ref{mainc} by using a base curve of higher genus.  More generally, in the case $X$ is not necessarily semistable, we have the inequality \cite[Cor. 2.7]{Shi2}
\begin{equation*}
\mu+2\alpha\geq 2\chi(\co_X)+2-2g,
\end{equation*}
where $\mu$ and $\alpha$ are the number of multiplicative and additive singular fibers, respectively.  This gives some hope of using additive singular fibers to improve the number $\frac{3}{2}$ in Corollary \ref{mainc}.  In this case, however, the singular fibers are no longer normal crossings divisors (as is required in Vojta's Conjecture), and this appears to prevent one, at least with our method, from improving on Corollary \ref{mainc}.

Using a trick of Vojta \cite[Prop. 5.4.1]{Vo}, we can use our results on elliptic surfaces to construct varieties $X$ such that $D=0$ and the exceptional set $Z(r,\epsilon,X,D,A)$ must depend on the parameter $r$.

\begin{theorem}
\label{D0}
There exists a nonsingular projective variety $V$ defined over a number field $k$, curves $C_i\subset V$, $i\in \mathbb{N}$, with $\cup_i C_i$ Zariski-dense in $V$, a big divisor $A$, and a function $g:\mathbb{N}\to \mathbb{N}$ such that for any $i\in \mathbb{N}$, $\delta>0$, sufficiently small $\epsilon>0$, and constant $c_i$,
\begin{equation*}
h_{K_V}(P)> \left(\frac{3}{2}-\delta\right)d(P)+\epsilon h_A(P)+c_i
\end{equation*}
for an infinite set of points $P\in C_i$ with $[k(P):k]\leq g(i)$.  In particular, for any set of constants $c_i$, $i\in \mathbb{N}$, $\delta>0$, and any sufficiently small $\epsilon>0$, the inequality
\begin{equation*}
h_{K_V}(P)> \left(\frac{3}{2}-\delta\right)d(P)+\epsilon h_A(P)+c_{[k(P):k]}
\end{equation*}
holds for a Zariski-dense set of points $P\in V(\kbar)$.
\end{theorem}

\begin{proof}
By Theorem \ref{Sh}, there exists a semistable elliptic surface $X$ over $\mathbb{P}^1$ with a section over a number field $k$ and $N=2\chi(\co_X)+2$ singular fibers.  Let $D$ be the sum of the singular fibers, which we may assume are all defined over $k$.  By Theorem~\ref{main} and its proof, there exist infinitely many distinct curves $C_i'\subset X$, $i\in \mathbb{N}$, with an infinite set of $D$-integral points $T_i$ of bounded degree, and an ample divisor $A'$, such that for any constant $c_i$ and sufficiently small $\epsilon>0$,
\begin{equation*}
m_{D,S}(P)+h_{K_X}(P)>\left(\frac{3}{2}-\delta\right)d(P)+\epsilon h_{A'}(P)+c_i,
\end{equation*}
for all but finitely many points $P$ in $T_i$.

We may assume one of the singular fibers lies over the point at infinity of $\mathbb{P}^1$.  Let $t_1,\ldots,t_{N-1}\in k$ be the projections of the other singular fibers.  Let $M$ be a large enough positive integer such that
\begin{equation*}
m_{D,S}(P)<\frac{M\epsilon}{2}h_{A'}(P)+O(1)
\end{equation*}
for all $P$ in $X(\kbar)\setminus D$.  Let $B$ be the nonsingular projective curve associated to $k(\sqrt[M]{t-t_1},\ldots,\sqrt[M]{t-t_{N-1}})$.  We have a natural morphism $\phi:B\to \mathbb{P}^1$ of degree $M^{N-1}$, with ramification index $M$ at each point above $t_1,\ldots,t_{N-1},\infty$ and unramified everywhere else.  Let $V=X\times_{\mathbb{P}^1} B$ and let $\pi:V\to X$ be the projection map.  From the above,  it follows easily that the ramification divisor $R$ of $\pi$ is given by $R=\frac{M-1}{M}\pi^*D$.  We also have the formula $K_V=\pi^*K_X+R$.  Since the set of points $T_i$ is $D$-integral, by Theorem \ref{dineq2} we have
\begin{equation*}
d(Q)\leq d(\pi(Q))+O(1)
\end{equation*}
for all points $Q$ in $\pi^{-1}(T_i)$.  Let $A=\pi^*A'$, a big divisor.  Then for all $Q\in \pi^{-1}(T_i)$, setting $P=\pi(Q)$, we have (up to $O(1)$)
\begin{align*}
h_{K_{V}}(Q)&=h_{\pi^*K_X}(Q)+h_R(Q)=h_{K_X}(P)+\frac{M-1}{M}h_D(P)\\
&\geq h_{K_X}(P)+\frac{M-1}{M}m_{D,S}(P)\\
&\geq h_{K_X}(P)+m_{D,S}(P)-\frac{\epsilon}{2} h_{A'}(P)\\
&\geq \left(\frac{3}{2}-\delta\right)d(P)+\frac{\epsilon}{2} h_{A'}(P)\\
&\geq \left(\frac{3}{2}-\delta\right)d(Q)+\frac{\epsilon}{2} h_{A}(Q).
\end{align*}
This yields the theorem, taking for each $i\in \mathbb{N}$, $C_i$ to be an appropriate curve in $\pi^{-1}(C_i')$.
\end{proof}

Finally, let us discuss some recent results in a related setting that are closely connected to the results of this section.  In a geometric setting, Autissier, Chambert-Loir, and Gasbarri having recently shown \cite{ACG}:
\begin{theorem}[Autissier, Chambert-Loir, Gasbarri]
\label{ACG}
Let $n\geq 2$ be an integer and $k$ an algebraically closed field of characteristic zero.  There exists a nonsingular projective variety $X$ over $k$ of dimension $n$, with canonical divisor $K$, and projective curves $C_i\subset X$, $i\in \mathbb{N}$, such that:
\begin{enumerate}
\item $\cup_{i\in \mathbb{N}}C_i$ is Zariski-dense in $X$.
\item Let $A$ be an ample divisor on $X$ and let $\delta>0$.  For all sufficiently small $\epsilon>0$,
\begin{equation*}
C_i.K> (n-\delta)\chi(C_i)+\epsilon C_i.A,
\end{equation*}
for all but finitely many $i\in \mathbb{N}$.
\end{enumerate}
\end{theorem}
Here, for a projective curve $C$, we let $\chi(C)=2g(C)-2$ where $g(C)$ is the geometric genus of $C$.  The terms $C.K$, $C.A$, and $\chi(C)$ are analogous to (the unnormalized versions of) $h_K(P)$, $h_A(P)$, and $d(P)$, respectively, in Conjecture \ref{Vgeneral}.  As noted in \cite{ACG}, the gonalities of the curves $C_i$ constructed in \cite{ACG} in Theorem \ref{ACG} go to infinity with $i$ (the gonality is an analogue of the degree $[k(P):k]$).  Theorem \ref{ACG} shows that in a geometric setting the analogue of \eqref{Vineq} does not hold uniformly without at least a $\dim X$ factor in front of (the analogue of) the $d(P)$ term.  

On the other hand, for nonsingular curves on minimal surfaces of general type, Miyaoka \cite{Miy} proved:
\begin{theorem}[Miyaoka]
Let $X$ be a nonsingular minimal complex projective surface of general type with canonical divisor $K$.  Let $A$ be an ample divisor on $X$ and $\epsilon>0$.  Then 
\begin{equation*}
C.K\leq \frac{3}{2}\chi(C)+\epsilon C.A+O(1),
\end{equation*}
for all nonsingular projective curves $C$ in $X$.
\end{theorem}
It is unclear if the appearance of the coefficient $\frac{3}{2}$ in both Miyaoka's result and the results given here is coincidence or holds some deeper significance.

\bibliography{VC}
\end{document}